\documentclass[preprint,11pt]{article}
\usepackage{comment}

\usepackage[english]{babel}
\usepackage[T1]{fontenc}
\usepackage[utf8]{inputenc}
\usepackage{amsmath,amssymb,mathrsfs,enumerate}
\usepackage{amscd}
\usepackage{a4wide}
\usepackage{tikz}
\usepackage{physics}

\date{}

\usepackage{hyperref}

\numberwithin{equation}{section}
\newtheorem{theo}{Theorem}[section]

\newtheorem{prop}{Proposition}[section]
\newtheorem{lem}{Lemma}[section]
\newtheorem{de}{Definition}[section]
\newtheorem{rem}{\textup{\textbf{Remark}}}[section]
\numberwithin{figure}{section}

\newcommand{\br}{{\mathbb{R}}}
\newcommand{\bc}{{\mathbb{C}}}

\newenvironment{proof}
	{\textit{Proof.}}
	{\hfill $\square$\vskip 8pt}

\title{Asymptotic estimates of solutions to time-fractional diffusion equations with space-dependent variable order}
\author{
Yavar {\sc Kian}$\footnote{Aix Marseille Univ., Universit\'e de Toulon, CNRS, CPT, Marseille, France, 
yavar.kian@univ-amu.fr}$, 
Diomba {\sc Sambou}$\footnote{Departamento de Matem\'aticas, Facultad de 
Matem\'aticas, Pontificia Universidad Cat\'olica de Chile, Vicu\~na Mackenna 
4860, Santiago de Chile, disambou@mat.uc.cl}$,
and \'Eric {\sc Soccorsi}$\footnote{Aix Marseille Univ., Universit\'e de Toulon, CNRS, CPT, Marseille, France, eric.soccorsi@univ-amu.fr}$
}

\begin{document}

\maketitle

\begin{abstract} 
We examine the short and long-time behaviors of time-fractional diffusion equations with variable space-dependent order. More precisely, we describe the time-evolution of the solution to these equations as the time parameter goes either to zero or to infinity.
\end{abstract}

\bigskip


\bigskip



\section{Introduction}



Let $K$ be a compact subset of $\br^d$ with $d \geq 2$, and let  $\kappa \in L^\infty(K,\br_+$. More precisely, we assume that $\kappa : K  \to [\alpha_m,\alpha_M]$, where $\alpha_m$ and $\alpha_M$ are two fixed constants such that $0<\alpha_m<\alpha_M<1$. Given $\alpha_* \in [\alpha_m,\alpha_M]$, we put
\begin{equation}
\label{eq:ha2}
\alpha(x) = \begin{cases} 
\kappa(x) & \text{if}\ x \in K, \\ 
\alpha_* & \text{if}\ x \in \br^d \setminus K,
\end{cases}
\end{equation}
in such a way that $\alpha \in L^\infty(\br^d)$ verifies
\begin{equation}
\label{eq:ha1}
0 < \alpha_m \le \alpha(x) \le \alpha_M < 1, \quad x \in \br^d.
\end{equation}

Given $u_0 : \br^d \to \mathbb{C}$, we consider the initial value problem (IVP) for the time-fractional diffusion equation with variable order (VO) $\alpha(x)$, 
\begin{equation}
\label{eq:vp}
\begin{cases} 
\partial_t^{\alpha(x)} u(t,x) - \Delta u(t,x) = 0, & (t,x) \in (0,+\infty) \times \mathbb{R}^d, \\ 
u(0,x) = u_0(x), & x \in \mathbb{R}^d.
\end{cases}
\end{equation}
Here $\partial_t^{\alpha(x)}$ denotes the Caputo fractional derivative with respect to $t$, of order $\alpha(x)$, defined by 
$$
\partial_t^{\alpha(x)} u(t,x) := \frac{1}{\Gamma(1 - \alpha(x))} \int_0^t (t - s)^{-\alpha(x)}
\partial_s u(s,x) ds, \quad (t,x) \in (0,+\infty) \times \mathbb{R}^d,
$$ 
where $\Gamma$ is the usual Gamma function. The existence and uniqueness of the solution to \eqref{eq:vp} in a bounded domain of $\br^d$ was proved in \cite{ksy} for suitable initial states $u_0$. The main purpose of this paper is to extend the above mentioned existence and uniqueness result to $\br^d$, and then describe the time asymptotic behavior of the solution to \eqref{eq:vp} as $t\to0$ and $t\to+\infty$.


\subsection{Motivations}

Anomalous diffusion in complex media has attracted much attention from the scientific community in the last decade, with many  applications in geophysics, environmental 
science and biology. The diffusion features of homogeneous media are 
often modeled by constant order (CO) 
time-fractional diffusion equations, corresponding to \eqref{eq:vp} with a constant function
$x \mapsto \alpha(x)$, see e.g. \cite{AG,CSLG}.

However, in several complex media where space inhomogeneous variations are generated by the presence of 
heterogeneous regions, the CO fractional dynamic models are no longer robust for long times, see \cite{FS}. In this context, the VO time-fractional models are more relevant for describing the space-dependent anomalous diffusion process, see e.g. \cite{SCC}. As a matter of fact, several VO diffusion models have been successfully applied to numerous areas of applied sciences and engineering, such as chemistry \cite{CZZ}, 
rheology \cite{SdV}, biology \cite{GN}, hydrogeology \cite{AON} and physics \cite{SS, ZLL}. 
We point out that the VO time-fractional kinetic equation is frequently considered as a macroscopic model to continuous time random walk (CRTW)  associated with stochastic diffusion processes with space-dependent diffusion coefficient, and we refer to \cite{ORT,RT} for the derivation of VO time-fractional diffusion models from a CTRW scheme.

The IVP  under consideration in this paper models anomalous diffusion in the whole space $\br^d$, with inhomogeneous perturbation induced by heterogeneous regions located in the compact subset $K$. In accordance with the above paragraph, the space inhomogeneous variations are described by the {\it a priori} non constant fractional power $\alpha$. The main purpose of this article is to describe how compactly supported inhomogeneous variations of the diffusion space affect short and long time behaviors of the diffusion process.


\subsection{Time-fractional diffusion equations: A short review of the mathematical literature}

The analysis of ordinary or partial differential equations with time-fractional derivatives has been growing in interest among the mathematics community over the last decade. Without being exhaustive, we refer to \cite{KST,P, SY} for an introduction to the mathematical approach of time-fractional differential equations. 

While constant order and distributed order (DO) fractional diffusion equations have been extensively studied (see e.g. \cite{KY,LKS, SY}) by several authors, the reference \cite{ksy} is, as far as we know, the only mathematical paper dedicated to the theoretical study of VO time-fractional diffusion equations. More specifically, the time asymptotic behavior of the solution to time-fractional differential equations in a bounded spatial domain was examined in \cite{SY} for CO fractional diffusion equations, and in \cite{LLY} for DO fractional diffusion equations. In this paper we aim to make the same for VO fractional diffusion equations in $\br^d$.


\subsection{Main result, brief comments and outline}

The main result of this article can be stated as follows.

\begin{theo}
\label{t1} 
Fix $s \in (\max(d,4),+\infty)$, let $s'\in(1,+\infty)$ be such that $\frac{1}{s}+\frac{1}{s'}=1$, and pick $u_0\in L^2(\br^d)\cap L^{\frac{2s'}{2s'-1}}(\br^d)$. Assume that $\alpha_* \in [\alpha_m,\alpha_M]$ fulfills $\alpha_* \left(1-\frac{d}{s} \right)<\alpha_m$ and that $\alpha \in L^\infty(\br^d)$ verifies \eqref{eq:ha2}-\eqref{eq:ha1}. Then there exists a constant $C>0$, depending only on $d$, $s$, $K$, $\alpha_m$, $\alpha_M$ and $\alpha_*$, such that the solution $u$ to \eqref{eq:vp} satisfies the two following estimates:
 \begin{equation}
\label{t1a} 
\|u(t,\cdot)\|_{L^2(\br^d)}\leq C\left(\|u_0\|_{L^2(\br^d)}+\|u_0\|_{L^{\frac{2s'}{2s'-1}}(\br^d)}\right) t^{-\left(\alpha_m-\alpha_* \left( 1-\frac{d}{s} \right) \right)},\ t \in (1,+\infty),
\end{equation}
and
\begin{equation}
\label{t1b} 
\|u(t,\cdot)\|_{L^2(\br^d)}\leq C\|u_0\|_{L^2(\br^d)}t^{\alpha_m-\alpha_M},\  t\in(0,1].
\end{equation}
\end{theo}

If $\alpha$ is constant in $\br^d$ then we have $\alpha_m=\alpha_M$ and in this particular case \eqref{t1b} coincides with the short time behavior of the solution for CO fractional diffusion equations, given in \cite[Theorem 2.1]{SY}. This result indicates  that a local perturbation of $\alpha$ induces a singularity of the solution to \eqref{eq:vp} at $t=0$.   

However, for large time parameters, the rate of decay given by \eqref{t1a} in the particular case where $\alpha_m=\alpha_M$, is weaker than the one computed in \cite[Theorem 2.1]{SY} for CO fractional diffusion equations. This is due to the additional technical requirement in \eqref{t1a} that $u_0 \in L^{\frac{2s'}{2s'-1}}(\br^d)$, which was needed in the derivation of \eqref{t1a}.

The paper is organized as follows. In Section 2, we prove the existence and uniqueness of the solution to \eqref{eq:vp} by adapting the strategy implemented in \cite{ksy}. Section 3 contains technical estimates for the operator valued function $p\mapsto (-\Delta+p^\alpha)p^\alpha$ that are useful for the proof of Theorem \ref{t1}, given in Section 4.

\section{Analysis of the forward problem}


In this section we  give a precise definition of a weak solution to \eqref{eq:vp}, which is inspired from \cite[Section 2.2]{ksy}. Then we
prove that this solution exists and is unique within a suitable class of functions.

\subsection{Weak solution}
Let us first introduce some notations. We denote by ${\mathcal S}'(\br,L^2(\br^d))$ the space dual to ${\mathcal S}(\br,L^2(\br^d))$ and by $${\mathcal S}'(\br_+,L^2(\br^d)) := \{u \in {\mathcal S}'(\br,L^2(\br^d));\ \mbox{supp}\ u \subset [0,+\infty) \times \br^d \}, $$ 
the set of tempered distributions in ${\mathcal S}'(\br,L^2(\br^d))$ that are supported in $[0,+\infty) \times \br^d$. Evidently,
$u \in {\mathcal S}'(\br_+,L^2(\br^d))$ if and only if 
$u \in {\mathcal S}'(\br,L^2(\br^d))$ verifies $\langle u , \varphi \rangle_{{\mathcal S}'(\br,L^2(\br^d)),{\mathcal S}(\br,L^2(\br^d))} =0$
whenever $\varphi \in {\mathcal S}(\br,L^2(\br^d))$ vanishes in $\br_+ \times \br^d$. 
Thus, for a.e. $x\in\br^d$ we have $\langle u(\cdot,x) , \varphi \rangle_{S'(\br),S(\br)}=\langle u(\cdot,x), 
\psi \rangle_{S'(\br),S(\br)}$, provided $\varphi \in \mathcal S(\br)$ and $\psi \in \mathcal S(\br)$ coincide in $\br_+$.  In light of this, we say that $\varphi \in \mathcal S(\br_+)$ if $\varphi$ is the 
restriction to $\br_+$ of a function $\tilde{\varphi}\in \mathcal S(\br)$, and we set
$$
x\mapsto\langle u(\cdot,x) , \varphi \rangle_{\mathcal S'(\br_+),\mathcal S(\br_+)} 
:= x\mapsto\langle u(\cdot,x)  , \tilde{\varphi} \rangle_{\mathcal S'(\br),\mathcal S(\br)},
\ u \in \mathcal S'(\br_+,L^2(\br^d)),
$$
where $\tilde{\varphi}$ is any function in $\mathcal S(\br)$ 
such that $\tilde{\varphi}(t)=\varphi(t)$ for all $t \in \br_+$.

Next, for all $p \in (0,+\infty)$, we put
$$ 
e_p(t) := \exp (-p t ),\ t \in \br_+.
$$
Since $e_p \in S(\br_+)$, we define the Laplace transform  with respect to $t$ of $u \in \mathcal S'(\br_+,L^2(\br^d))$, as
\begin{equation}
\label{def-L}
U(p,\cdot) :  x\mapsto\langle u(\cdot,x) , e_p \rangle_{\mathcal S'(\br_+),\mathcal S(\br_+)},\ p \in (0,+\infty).
\end{equation}
Having said that, we may now define a weak solution to \eqref{eq:vp} as follows.

\begin{de}
\label{d1} 
Given $u_0 \in L^2(\br)$, we say that $u$ is a weak solution to \eqref{eq:vp} if $u \in \mathcal S'(\br_+,L^2(\br^d))$ has its Laplace transform $U$ with respect to $t$, defined in \eqref{def-L}, such that \begin{equation}
\label{a3}
 (- \Delta+p^{\alpha(x)}) U(p,x) = p^{\alpha(x)-1}u_0(x),\quad (p,x) \in (0,+\infty) \times \mathbb{R}^d.
\end{equation} 
\end{de}

\subsection{Existence and uniqueness result}

Assume that $u_0 \in L^2(\br^d)$. Then upon arguing as in the derivation of \cite[Theorem 1.1]{ksy}, we obtain that the IVP \eqref{eq:vp} admits a unique solution $u \in {\mathcal C}((0,T],L^2(\br^d))$, enjoying the following Duhamel representation formula
\begin{equation}
\label{sol}
u(t,x) = \frac{1}{2i\pi} \int_{\gamma(\epsilon,\theta)} e^{tp} \big( -\Delta + p^{\alpha(x)} \big)^{-1} p^{\alpha(x) - 1} u_0(x) dp.
\end{equation}
Here $\theta \in (\pi \slash 2,\pi )$ and $\epsilon \in (0,+\infty)$ are arbitrary, and
\begin{equation}
\label{g1} 
\gamma(\epsilon,\theta) := \gamma_-(\epsilon,\theta) \cup \gamma_0(\epsilon,\theta) \cup \gamma_+(\epsilon,\theta),
\end{equation}
is the contour in $\bc$, associated with
\begin{equation}
\label{g2}
\gamma_0(\epsilon,\theta) := \{ \epsilon e^{i \beta};\ \beta \in[-\theta,\theta] \}\ \mbox{and}\ \gamma_\pm(\epsilon,\theta) := \{s e^{\pm i \theta};\ s \in [\epsilon,+\infty) \}.
\end{equation}

The proof of \eqref{sol} follows the same path as the derivation \cite[Theorem 1.1]{ksy}: It relies on a careful application of the Bromwich-Mellin formula, which boils down to the following basic resolvent estimates, directly inspired by \cite[Proposition 2.1]{ksy}.

\begin{prop}
\label{p:re4}
Assume that the function $\alpha$ fulfills \eqref{eq:ha2}-\eqref{eq:ha1}. 
\begin{enumerate}
\item If $\kappa$ is not identically equal to $\alpha_*$ in $K$, we have
\begin{equation}
\label{eq:es3}
\big\Vert \big( -\Delta + r^{\alpha(x)} e^{i \alpha(x) \beta} \big)^{-1} \big\Vert_{\mathcal{B}(L^2(\br^d))} \le C_\beta \max_{j=m,M} r^{-\alpha_j},\ r \in (0,+\infty),\ \beta \in (-\pi,\pi),
\end{equation}
with
$$
C_\beta := \begin{cases} 
1 & \text{if} \quad  \beta = 0, \\ 
\max_{j=m,M} \big\vert \sin(\alpha_j\beta) \big\vert^{-1} & \text{if} \quad \beta \in \pm (0,\pi).
\end{cases}
$$
Moreover, the mapping $p \mapsto \big( -\Delta + p^{\alpha(x)} \big)^{-1}$ is 
bounded holomorphic in $\bc \setminus \br_-$.
\item If $\kappa$ is constant and equals $\alpha_*$ in $K$ (hence $\alpha(x)=\alpha_*$ for all $x \in \br^d$), we have
\begin{equation}
\label{eq:es3c}
\big\Vert \big( -\Delta + r^\alpha_* e^{i \alpha_*\beta}\big)^{-1} \big\Vert_{\mathcal{B}(L^2(\br^d))} \le C_\beta r^{-\alpha_*},\ r \in (0,+\infty),\ \beta \in (-\pi,\pi),
\end{equation}
with
$$
C_\beta := \begin{cases} 
1 & \text{if} \quad  \beta = 0, \\ 
\big\vert \sin(\alpha_*\beta) \big\vert^{-1} & \text{if} \quad \pm\beta \in  (0,\pi).
\end{cases}
$$ 
Moreover, the mapping $p \mapsto \big( -\Delta + p^{\alpha_*} \big)^{-1}$ is bounded holomorphic in $\bc \setminus \br_-$.

\end{enumerate}
\end{prop}

\noindent
\begin{proof}
We shall prove the first statement only, the second one being obtained in a similar way. Moreover, since both cases $r \in (0,1)$ and
 $r \in [1,+\infty)$ can be handled analogously, we assume that $r \geq 1$ in the remaining part of the proof.

a) Suppose that $\beta \in (0,\pi)$, the case $\beta \in (-\pi,0)$ being treated in the same way. Let $U_\beta$ denote the multiplication operator in 
$L^2(\br^d)$ by the function
$$
u_\beta(x) := \Big( r^{\alpha(x)} \sin \big( \beta \alpha(x) \big) \Big)^{1/2}, \quad
x \in \br^d. 
$$
Notice that $iU_\beta^2$ is the skew-adjoint part of the operator $-\Delta + r^{\alpha(x)} e^{i\alpha(x)\beta}$. By setting $m_\beta := \min_{j=m,M} \sin(\alpha_j\beta)$, 
we have
$$
0 < m_\beta^{1/2}r^{\alpha_m/2} \le u_\beta(x) \le r^{\alpha_m/2}, \quad x \in \br^d,
$$
so the self-adjoint operator $U_\beta$ is bounded and invertible in $L^2(\br^d)$, and
\begin{equation}
\label{eq:re5}
\Vert U_\beta^{-1} \Vert_{\mathcal{B}(L^2(\br^d))} \le m_\beta^{-1/2}r^{-\alpha_m/2}.
\end{equation}
For any $p = re^{i\beta}$, let $B_p := -\Delta + r^{\alpha(x)} \cos \big( \beta \alpha(x) \big)$ denote the self-adjoint part of the operator 
$-\Delta + r^{\alpha(x)} e^{i\alpha(x)\beta}$, i.e.
$$
-\Delta + r^{\alpha(x)} e^{i\alpha(x)\beta} = B_p + iU_\beta^2.
$$
Since the multiplication operator by the function $x \mapsto r^{\alpha(x)} \cos \big( \beta \alpha(x) \big)$ is bounded (by $r^{\alpha_M}$) in $L^2(\br^d)$, then 
the operator $B_p$ is self-adjoint in $L^2(\br^d)$ with domain $D(B_p) = D(-\Delta)$. Thus, $U_\beta^{-1} B_p U_\beta^{-1}$ is self-adjoint in $L^2(\br^d)$ 
with domain $U_\beta D(-\Delta)$ and we have
\begin{equation}
\label{eq:re6}
-\Delta + p^{\alpha(x)} = U_\beta \big( U_\beta^{-1} B_p U_\beta^{-1} + i) U_\beta,
\end{equation}
with 
\begin{equation}
\label{eq:re7}
\Vert \big( U_\beta^{-1} B_p U_\beta^{-1} + i)^{-1} \Vert_{\mathcal{B}(L^2(\br^d))} \le 1.
\end{equation}
From this and \eqref{eq:re6} it follows that the operator 
$-\Delta + r^{\alpha(x)} e^{i\alpha(x)\beta}$ is invertible in $L^2(\br^d)$, with
\begin{equation}
\label{eq:re8}
\big( -\Delta + r^{\alpha(x)} e^{i\alpha(x)\beta} \big)^{-1} = U_\beta^{-1} \big( U_\beta^{-1} B_p U_\beta^{-1} + i)^{-1} U_\beta^{-1}.
\end{equation}
Hence $\big( -\Delta + r^{\alpha(x)} e^{i\alpha(x)\beta} \big)^{-1}$ maps 
$L^2(\br^d)$ onto $U_\beta^{-1} D \big( U_\beta^{-1} B_p U_\beta^{-1} \big) = D(-\Delta)$, and we infer from
\eqref{eq:re5} and \eqref{eq:re7}-\eqref{eq:re8} that
$$
\Vert \big( -\Delta + r^{\alpha(x)} e^{i\alpha(x)\beta} \big)^{-1} \Vert_{\mathcal{B}(L^2(\br^d))} \le \Vert \big( U_\beta^{-1} B_p U_\beta^{-1} + i)^{-1} \Vert_{\mathcal{B}(L^2(\br^d))} \Vert U_\beta^{-1} \Vert_{\mathcal{B}(L^2(\br^d))}^2 \le m_\beta^{-1}r^{-\alpha_m}.
$$

b) Assume that $\beta = 0$. Then we have $-\Delta + p^{\alpha(x)} = -\Delta + r^{\alpha(x)}$ and consequently
$$
-\Delta + p^{\alpha(x)} \ge r^{\alpha_m} > 0,
$$
in the operator sense. Therefore $-\Delta + p^{\alpha(x)}$ is boundedly invertible in $L^2(\br^d)$, with
$$
\big\Vert \big( -\Delta + p^{\alpha_*} \big)^{-1} \big\Vert_{\mathcal{B}(L^2(\br^d))} \le r^{-\alpha_m}.
$$

Finally, arguing as in the derivation \cite[Proposition 2.1]{ksy}, we get that $p \mapsto \big( -\Delta + p^{\alpha(x)} \big)^{-1}$ is bounded holomorphic in $\bc \setminus \br_-$, so the proof is complete.
\end{proof}

It is clear from \eqref{sol} that the time evolution of the solution $u$ to \eqref{eq:vp} is determined by the behavior of the mapping
$p \mapsto ( -\Delta + p^{\alpha(x)})^{-1} p^{\alpha(x)} u_0(x)$ on the integration path $\gamma(\epsilon,\theta)$. More precisely, we shall see that the analysis of the time asymptotic behavior of $u$ boils down to specific estimates of the resolvent of the operator $-\Delta + p^{\alpha(x)}$, given in the coming section.

\section{Resolvent estimates}

For  $s \in [1,+\infty)$ we denote by ${\mathfrak{S}^s(L^2(\br^d))}$ the Schatten-von Neumann classes of compact linear operators $L$ for which the 
norm $\Vert L \Vert_{\mathfrak{S}^s(L^2(\br^d))} := \big( \textup{Tr} \; \vert L \vert^s \big)^{1 \slash s}$ is finite. For further use, we recall the two following 
well known properties (see e.g. \cite{S}) about these sets:
\begin{enumerate}[(SvN1)]
\item For all $L \in {\mathfrak{S}^s(L^2(\br^d))}$ it holds true that $\Vert L \Vert_{\mathcal{B}(L^2(\br^d))} \le \Vert L \Vert_{\mathfrak{S}^s(L^2(\br^d))}$.
\item For all $A \in \mathcal{B}(L^2(\br^d))$ and all $L \in {\mathfrak{S}^s(L^2(\br^d))}$, the two operators $AL$ and $LA$ belong to ${\mathfrak{S}^s(L^2(\br^d))}$. 
Moreover the norms $\Vert AL \Vert_{\mathfrak{S}^s(L^2(\br^d))}$ and $\Vert LA \Vert_{\mathfrak{S}^s(L^2(\br^d))}$ are majorised by $\Vert A \Vert_{\mathcal{B}(L^2(\br^d))} \Vert L \Vert_{\mathfrak{S}^s(L^2(\br^d))}$.
\end{enumerate}

We start by establishing the following technical result.

\begin{lem}
\label{l:re}
Fix $s \in \left(  \max(d,4) , +\infty \right)$, pick 
$\alpha_* \in [\alpha_m,\alpha_M]$ such that $\alpha_* \left( 1 - \frac{d}{s} \right)< \alpha_m$, and set
\begin{equation}
\label{eq:ct}
\theta := \min \Big( \frac{2\pi}{3},\frac{\pi}{2}\alpha_*^{-1} \Big).
\end{equation}
Let $\alpha \in L^\infty(\br^d)$ fulfill \eqref{eq:ha2}-\eqref{eq:ha1}.
Then, there exists a constant $r_0=r_0(d,s,\alpha_m,\alpha_*) \in (0,1)$ such for all $p=r e^{i \beta}$ with $r \in (0,r_0)$ 
and $\beta \in [ -\theta,\theta ]$, we have:
\begin{equation}
\label{eq:es1}
\big\Vert \big( p^{\alpha_*} - p^{\alpha} \big) \big( -\Delta + p^{\alpha_*} \big)^{-1} \big\Vert_{\mathcal{B}(L^2(\br^d))} \le \frac{1}{2}.
\end{equation}
\end{lem}

\begin{proof}
Since $p^{\alpha_*} - p^{\alpha(x)} = \big( p^{\alpha_*} - p^{\alpha(x)} \big) {\bf 1}_K(x)$ for all $x \in \br^d$, 
where ${\bf 1}_K$ denotes the characteristic function of $K$, from \eqref{eq:ha2}, we have
\begin{eqnarray}
\big\Vert  \big( p^{\alpha_*} - p^{\alpha} \big) \big( -\Delta + p^{\alpha_*} \big)^{-1} \big\Vert_{\mathcal{B}(L^2(\br^d))} & \leq &
\sup_{x \in \br^d} \big\vert p^{\alpha_*} - p^{\alpha(x)} \big\vert \big\Vert {\bf 1}_K \big( -\Delta + p^{\alpha_*} \big)^{-1} \big\Vert_{\mathcal{B}(L^2(\br^d))} \nonumber \\
& \leq & 2 r^{\alpha_m} \big\Vert {\bf 1}_K \big( -\Delta + p^{\alpha_*} \big)^{-1} \big\Vert_{\mathcal{B}(L^2(\br^d))}.\label{eq:re01}
\end{eqnarray}
To estimate the last factor of the product in the right hand side of \eqref{eq:re01}, we successively apply the above Properties (SvN1) and (SvN2), and obtain that
\begin{eqnarray}
& & \big\Vert {\bf 1}_K \big( -\Delta + p^{\alpha_*} \big)^{-1} \big\Vert_{\mathcal{B}(L^2(\br^d))} \nonumber \\
& \le &
\big\Vert {\bf 1}_K \big( -\Delta + p^{\alpha_*} \big)^{-1} \big\Vert_{\mathfrak{S}^{s/2}(L^2(\br^d))} \nonumber \\
& \le & \big\Vert {\bf 1}_K \big( -\Delta + r^{\alpha_*} \big)^{-1} \big\Vert_{\mathfrak{S}^{s/2}(L^2(\br^d))} \big\Vert \big( -\Delta + r^{\alpha_*} \big) 
\big( -\Delta + p^{\alpha_*} \big)^{-1} \big\Vert_{\mathcal{B}(L^2(\br^d))}. \label{eq:re03}
\end{eqnarray}
Next, since $\frac{s}{2} \ge 2$, we have
$$
\big\Vert {\bf 1}_K \big( -\Delta + r^{\alpha_*} \big)^{-1} \big\Vert_{\mathfrak{S}^{s/2}(L^2(\br^d))} \le (2\pi)^{-\frac{4d}{s^2}} \Vert {\bf 1}_K\Vert_{L^{s/2}(\br^d)} 
\big\Vert \big( \vert \cdot \vert^2 + r^{\alpha_*} \big)^{-1} \big\Vert_{L^{s/2}(\br^d)},
$$
by virtue of  \cite[Theorem 4.1]{S}, with
$\big\Vert \big( \vert \cdot \vert^2 + r^{\alpha_*} \big)^{-1} \big\Vert^{s/2}_{L^{s/2}(\br^d)} 
= C_d \int_0^\infty \frac{\vartheta^{d-1}}{\big( \vartheta^2 + r^{\alpha_*} \big)^{s/2}} d\vartheta$ for some constant $C_d$ depending only on $d$, and
$\int_0^\infty \frac{\vartheta^{d-1}}{\big( \vartheta^2 + r^{\alpha_*} \big)^{s/2}} d\vartheta = \left( \int_0^\infty \frac{\varrho^{d-1}}{\big( \varrho^2 + 1 \big)^{s/2}} d\varrho \right) r^{-\alpha_*\big(\frac{s}{2} - \frac{d}{2}\big)}$. Therefore we get
\begin{equation}
\label{eq:re05}
\big\Vert {\bf 1}_K \big( -\Delta + r^{\alpha_*} \big)^{-1} \big\Vert_{\mathfrak{S}^{s/2}(L^{s/2}(\br^d))} \le C_{d,s} \Vert {\bf 1}_K \Vert_{L^{s/2}(\br^d)} r^{-\alpha_*\big(1 - \frac{d}{s}\big)},
\end{equation}
where $C_{d,s}$ is a positive constant depending only on $d$ and $s$.
Further we estimate the second term of the product in the right hand side of \eqref{eq:re03}, with the help of the spectral mapping theorem, getting:
\begin{equation}
\label{eq:re06}
\big\Vert \big( -\Delta + r^{\alpha_*} \big) \big( -\Delta + p^{\alpha_*} \big)^{-1} \big\Vert_{\mathcal{B}(L^2(\br^d))} \le 
\sup_{\lambda \in [0,+\infty)} \Bigg\vert \frac{\lambda + r^{\alpha_*}}{\lambda + p^{\alpha_*}} \Bigg\vert.
\end{equation}
Moreover, since $\beta \in \big[ -\theta,\theta \big]$ where $\theta$ is given by \eqref{eq:ct}, then $\alpha_* \beta \in \big[ -\frac{\pi}{2},\frac{\pi}{2} \big]$ and thus the real part of $p^{\alpha_*}$ is nonnegative: $\Re \big( p^{\alpha_*} \big) \ge 0$. As a consequence we have $$
\Bigg\vert \frac{\lambda + r^{\alpha_*}}{\lambda + p^{\alpha_*}} \Bigg\vert \le 1 + \Bigg\vert \frac{r^{\alpha_*} - p^{\alpha_*}}{\lambda + p^{\alpha_*}} \Bigg\vert 
\le 1 + \frac{2r^{\alpha_*}}{\big\vert \lambda + p^{\alpha_*} \big\vert} 
\le 1 + \frac{2r^{\alpha_*}}{r^{\alpha_*}} = 3,$$
for all $\lambda \in [0,+\infty)$, and \eqref{eq:re06} then yields
\begin{equation}
\label{eq:re07}
\big\Vert \big( -\Delta + r^{\alpha_*} \big) \big( -\Delta + p^{\alpha_*} \big)^{-1} \big\Vert_{\mathcal{B}(L^2(\br^d))} \le 3.
\end{equation}
Now, putting \eqref{eq:re03}, \eqref{eq:re05} and \eqref{eq:re07} together, we obtain
$$
\big\Vert {\bf 1}_K \big( -\Delta + p^{\alpha_*} \big)^{-1} \big\Vert_{\mathcal{B}(L^2(\br^d))} \le C_{d,s} 
\Vert {\bf 1}_K \Vert_{L^{s/2}(\br^d)} r^{-\alpha_*\big(1 - \frac{d}{s}\big)},
$$
by substituting $C_{d,s}$ for $3 C_{d,s}$. From this and \eqref{eq:re01} it then follows upon replacing $C_{d,s}$ by $2 C_{d,s}$, that
$$
 \big\Vert \big( p^{\alpha_*} - p^{\alpha} \big) \big( -\Delta + p^{\alpha_*} \big)^{-1} \big\Vert_{\mathcal{B}(L^2(\br^d))} 
\le C_{d,s} \Vert {\bf 1}_K \Vert_{L^{s/2}(\br^d)} r^{\alpha_m - \alpha_*\big(1 - \frac{d}{s}\big)}.
$$
Finally, the desired result follows from this by taking
$$ r_0 = \min \left( 1 , \left( 2 C_{d,s} \Vert {\bf 1}_K \Vert_{L^{s/2}(\br^d)} \right)
^{-\frac{1}{\alpha_m - \alpha_* \left(1 - \frac{d}{s}\right)}} \right). $$
\end{proof}

\begin{rem}
It can be seen from the proof that the result of Lemma \ref{l:re} remains valid upon replacing the constant $\frac{2\pi}{3}$ appearing in \eqref{eq:ct} by any arbitrary real number 
$\rho \in \left( \frac{\pi}{2} , \pi \right)$.
\end{rem}

Armed with Lemma \ref{l:re} we may now prove the main result of this section.

\begin{prop}
\label{p:re1}
Let $s$ and $\theta$ be as in Lemma \ref{l:re}, let $\alpha \in L^\infty(\br^d)$ fulfill \eqref{eq:ha2}-\eqref{eq:ha1}, and pick $u_0 \in L^1(\br^d) \cap L^2(\br^d)$. 
\begin{enumerate}[a)]
\item If $\alpha_*$ obeys the same conditions as in Lemma \ref{l:re} and that the function $\alpha$ is non constant in $\br^d$, then there exists a positive constant $C=C(d,s,K)$ such that the following estimate
\begin{equation}
\label{eq:es2}
\big\Vert \big( -\Delta + p^{\alpha} \big)^{-1} p^{\alpha} u_0 \big\Vert_{L^2(\br^d)}
\le C_{d,s,K} \Big( \Vert u_0 \Vert_{L^2(\br^d)} + \Vert \widehat{u_0} \Vert_{L^{2s'}(\br^d)} \Big) r^{\alpha_m-\alpha_*\big(1 - \frac{d}{s}\big)}
\end{equation}
holds for all $p = re^{i\beta} \in \bc \setminus \br_-$ with $\beta \in \big[ -\theta,\theta \big]$ and $r \in \big(0,r_0 \big)$, where $r_0$ is the constant defined in Lemma \ref{l:re}. Here $s' := \frac{s}{s-4}$ and $\widehat{u_0}$ is the Fourier tranform of $u_0$.
\item If $\alpha=\alpha_*$ a.e. in $\br^d$ for some $\alpha_* \in [\alpha_m,\alpha_M]$, then there exists a constant $C=C(d,s)>0$ such that we have
\begin{equation}
\label{eq:es2c}
\big\Vert \big( -\Delta + p^{\alpha_*} \big)^{-1} p^{\alpha_*} u_0 \big\Vert_{L^2(\br^d)} 
\le C_{d,s} \Vert \widehat{u_0} \Vert_{L^{2s'}(\br^d)} r^{\alpha_* \frac{d}{s}},
\end{equation}
for all $p = re^{i\beta} \in \bc \setminus \br_-$ 
with $\beta \in \big[ -\theta,\theta \big]$ and $r \in (0,1)$, $s'$ being the same as above.
\end{enumerate}
\end{prop}

\begin{proof}
We start by proving the first statement of the result, i.e. the one associated with a function $\alpha$ that is non constant in $\br^d$. To do that we use the following basic inequality
\begin{eqnarray}
& & \big\Vert  \big( -\Delta + p^{\alpha} \big)^{-1} p^{\alpha} u_0 \big\Vert_{L^2(\br^d)} \nonumber \\
& \le & \big\Vert \big( -\Delta + p^{\alpha} \big)^{-1} \big( p^{\alpha} - p^{\alpha_*} \big) u_0 \big\Vert_{L^2(\br^d)} + 
\big\Vert \big( -\Delta + p^{\alpha} \big)^{-1} p^{\alpha_*} u_0 \big\Vert_{L^2(\br^d)}, \label{eq:1}
\end{eqnarray}
and estimate each of the two terms appearing in the right hand side of \eqref{eq:1}. The first one is treated with the aid of the resolvent identity, 
$$
\big( -\Delta + p^{\alpha} \big)^{-1} - \big( -\Delta + p^{\alpha_*} \big)^{-1} 
= \big( -\Delta + p^{\alpha_*} \big)^{-1} \big( p^{\alpha_*} - p^{\alpha} \big) \big( -\Delta + p^{\alpha} \big)^{-1},
$$
giving
\begin{equation}
\label{eq:re1}
\Big( I - \big( -\Delta + p^{\alpha_*} \big)^{-1} \big( p^{\alpha_*} - p^{\alpha} \big) \Big) \big( -\Delta + p^{\alpha} \big)^{-1}  
= \big( -\Delta + p^{\alpha_*} \big)^{-1}.
\end{equation}
Since $\big\Vert \big( -\Delta + p^{\alpha_*} \big)^{-1} \big( p^{\alpha_*} - p^{\alpha} \big) \big\Vert_{\mathcal{B}(L^2(\br^d))} = 
\big\Vert \big( \overline{p}^{\alpha_*} - \overline{p}^{\alpha(x)} \big) \big( -\Delta + \overline{p}^{\alpha_*} \big)^{-1} \big\Vert_{\mathcal{B}(L^2(\br^d))}$, 
we infer from Lemma \ref{l:re} that the operator $I - \big( -\Delta + p^{\alpha_*} \big)^{-1} \big( p^{\alpha_*} - p^{\alpha} \big)$ is boundedly invertible in $L^2(\br^d)$, with 
\begin{equation}
\label{eq:re2}
\Big\Vert \Big( I - \big( -\Delta + p^{\alpha_*} \big)^{-1} \big( p^{\alpha_*} - p^{\alpha} \big) \Big)^{-1} \Big\Vert_{\mathcal{B}(L^2(\br^d))} 
\le \frac{1}{1 - \frac{1}{2}} \le 2.
\end{equation}
Therefore we have
\begin{equation}
\label{eq:re2,0}
\big( -\Delta + p^{\alpha} \big)^{-1} = \Big( I - \big( -\Delta + p^{\alpha_*} \big)^{-1} \big( p^{\alpha_*} - p^{\alpha} \big) \Big)^{-1} 
\big( -\Delta + p^{\alpha_*} \big)^{-1},
\end{equation}
from \eqref{eq:re1}, and
\begin{eqnarray*}
\big\Vert \big( -\Delta + p^{\alpha} \big)^{-1} \big( p^{\alpha} - p^{\alpha_*} \big) u_0 \big\Vert_{L^2(\br^d)} 
& \le  & 2 \big\Vert \big( -\Delta + p^{\alpha_*} \big)^{-1} \big( p^{\alpha} - p^{\alpha_*} \big) u_0 \big\Vert_{L^2(\br^d)} \nonumber \\
& \le & 2 \big\Vert \big( -\Delta + p^{\alpha_*} \big)^{-1} {\bf 1}_K \big\Vert_{\mathcal{B}(L^2(\br^d))} 
\big\Vert \big( p^{\alpha} - p^{\alpha_*} \big) u_0 \big\Vert_{L^2(\br^d)} 
\end{eqnarray*}
upon recalling that $p^{\alpha(x)} - p^{\alpha_*} = {\bf 1}_K(x) (p^{\alpha(x)} - p^{\alpha_*})$ for all $x \in \br^d$. 
Since $r \in (0,r_0) \subset (0,1)$, this leads to
\begin{eqnarray}
\big\Vert \big( -\Delta + p^{\alpha} \big)^{-1} \big( p^{\alpha} - p^{\alpha_*} \big) u_0 \big\Vert_{L^2(\br^d)}
& \le & 4 \big\Vert \big( -\Delta + p^{\alpha_*} \big)^{-1} {\bf 1}_K \big\Vert_{\mathcal{B}(L^2(\br^d))} \Vert u_0 \Vert_{L^2(\br^d)} r^{\alpha_m} \nonumber \\
& \le & 4 \big\Vert {\bf 1}_K \big( -\Delta + {\overline p}^{\alpha_*} \big)^{-1} \big\Vert_{\mathcal{B}(L^2(\br^d))} \Vert u_0 \Vert_{L^2(\br^d)} r^{\alpha_m}.\label{eq:re3}
\end{eqnarray}
Further, applying \eqref{eq:re03}-\eqref{eq:re05} and \eqref{eq:re07} with ${\overline p}$ instead of $p$, we  get that
$$
\big\Vert {\bf 1}_K \big( -\Delta + {\overline p}^{\alpha_*} \big)^{-1}  \big\Vert_{\mathcal{B}(L^2(\br^d))} \le C_{d,s} 
\Vert {\bf 1}_K \Vert_{L^{s/2}(\br^d)} r^{-\alpha_*\big(1 - \frac{d}{s}\big)},
$$
which together with \eqref{eq:re3}, yields
\begin{equation}
\label{eq:re4,00}
\big\Vert \big( -\Delta + p^{\alpha}  \big)^{-1} \big( p^{\alpha(x)} - p^{\alpha_*} \big) u_0 \big\Vert_{L^2(\br^d)}  \le C_{d,s}  \Vert {\bf 1}_K \Vert_{L^{s/2}(\br^d)} \Vert u_0 \Vert_{L^2(\br^d)} r^{\alpha_m-\alpha_*\big(1 - \frac{d}{s}\big)}, 
\end{equation}
upon substituting $4C_{d,s}$ for $C_{d,s}$.

We turn now to estimating the second term in the right hand side of \eqref{eq:1}, 
\begin{equation}
\label{eq-a}
\big\Vert \big( -\Delta + p^{\alpha} \big)^{-1} p^{\alpha_*} u_0 \big\Vert_{L^2(\br^d)} = 
\big\Vert \big( -\Delta + p^{\alpha} \big)^{-1} u_0 \big\Vert_{L^2(\br^d)} r^{\alpha_*}. 
\end{equation}
With reference to \eqref{eq:re2} -\eqref{eq:re2,0} we have
$\big\Vert \big( -\Delta + p^{\alpha} \big)^{-1} u_0 \big\Vert_{L^2(\br^d)}  \le 
2 \big\Vert \big( -\Delta + p^{\alpha_*} \big)^{-1} u_0 \big\Vert_{L^2(\br^d)}$, 
so we infer from the estimate
$$
\big\Vert \big( -\Delta + p^{\alpha_*} \big)^{-1} u_0 \big\Vert_{L^2(\br^d)} 
\le \big\Vert \big( -\Delta + p^{\alpha_*} \big)^{-1} \big( -\Delta + r^{\alpha_*} \big) \big\Vert_{\mathcal{B}(L^2(\br^d))} 
\big\Vert \big( -\Delta + r^{\alpha_*} \big)^{-1} u_0 \big\Vert_{L^2(\br^d)}
$$
and \eqref{eq:re07}, that
$$
\big\Vert \big( -\Delta + p^{\alpha} \big)^{-1} u_0 \big\Vert_{L^2(\br^d)} \le 6 \big\Vert \big( -\Delta + r^{\alpha_*} \big)^{-1} u_0 \big\Vert_{L^2(\br^d)}.
$$
From this, \eqref{eq-a} and the Plancherel formula it then follows that
\begin{equation}
\label{eq:re4,02}
\big\Vert \big( -\Delta + p^{\alpha} \big)^{-1} p^{\alpha_*} u_0 \big\Vert_{L^2(\br^d)}
\le 6 \big\Vert \big( \vert \cdot \vert^2 + r^{\alpha_*} \big)^{-1} \widehat{u_0} \big\Vert_{L^2(\br^d)} r^{\alpha_*}.
\end{equation}
We are thus left with the task of estimating $\big\Vert \big( \vert \cdot \vert^2 + r^{\alpha_*} \big)^{-1} \widehat{u_0} \big\Vert_{L^2(\br^d)}$. This can be achieved upon using the fact that
$u_0 \in L^1(\br^d) \cap L^2(\br^d)$, which entails $\widehat{u_0} \in L^2(\br^d) \cap L^\infty(\br^d)$ and consequently $\widehat{u_0} \in L^q(\br^d)$ for 
any $q \ge 2$. More precisely, for all $q > \frac{d}{4}$ and $q'$ such that $\frac{1}{q} + \frac{1}{q'} = 1$, we have 
\begin{eqnarray*}
\big\Vert \big( \vert \cdot \vert^2 + r^{\alpha_*} \big)^{-1} \widehat{u_0} \big\Vert_{L^2(\br^d)} 
&  = & \Bigg( \int_{\br^d} \frac{\vert \widehat{u_0}(\xi) \vert^2}{\big( \vert \xi \vert^2 + r^{\alpha_*} \big)^{2}}  d\xi \Bigg)^{\frac{1}{2}} \\
& \le &  \Bigg( \int_{\br^d} \vert \widehat{u_0}(\xi) \vert^{2q'} d\xi \Bigg)^{\frac{1}{2q'}} \Bigg( \int_{\br^d} \frac{d\xi}{\big( \vert \xi \vert^2 + r^{\alpha_*} \big)^{2q}} \Bigg)^{\frac{1}{2q}} \\
& \le & C_{d,q} \Vert \widehat{u_0} \Vert_{L^{2q'}(\br^d)} r^{ -\alpha_* \frac{4q-d}{4q}},
\end{eqnarray*}
by the H\"older inequality, the constant $C_{d,q}>0$ depending only on $d$ and $q$.
In the particular case where $q = \frac{s}{4}$ and $q' = \frac{s}{s - 4} = s'$, this leads to
$$
\big\Vert \big( \vert \cdot \vert^2 + r^{\alpha_*} \big)^{-1} \widehat{u_0} \big\Vert_{L^2(\br^d)} 
\le C_{d,s} \Vert \widehat{u_0} \Vert_{L^{2s'}(\br^d)} r^{-\alpha_* \frac{s-d}{s}},
$$
which together with \eqref{eq:re4,02}, yield
\begin{equation}
\label{eq:re4,5}
\big\Vert \big( -\Delta + p^{\alpha} \big)^{-1} p^{\alpha_*} u_0 \big\Vert_{L^2(\br^d)} \le C_{d,s} \Vert \widehat{u_0} \Vert_{L^{2s'}(\br^d)} r^{\alpha_* \frac{d}{s}}.
\end{equation}
As $r^{\alpha_m-\alpha_*\big(1 - \frac{d}{s}\big)} > r^{\alpha_* \frac{d}{s}}$ for all $r \in (0,1)$, we deduce \eqref{eq:es2} directly from \eqref{eq:1}, \eqref{eq:re4,00} and \eqref{eq:re4,5}.

Finally the second statement is a byproduct of \eqref{eq:re4,5} with $\alpha(x)=\alpha_*$ for all $x \in \br^d$.
\end{proof}


\section{Proof of Theorem \ref{t1}}

In this proof, $C$ is a generic constant depending only on $d$, $K$, $s$, $\alpha_m$, $\alpha_M$ and $\alpha_*$, that may change from line to line.

We start by showing \eqref{t1a}. To this end, we assume that $u_0\in L^1(\br^d)\cap L^2(\br^d)$. With reference to \eqref{sol} and \eqref{g1}-\eqref{g2}, we set for all $t \in (1,+\infty)$ and a.e. $x \in \br^d$,
\begin{equation}
\label{eq-u}
u^j(t,x):= \frac{1}{2i\pi} \int_{\gamma_j(t^{-1},\theta)} e^{tp} \big( -\Delta + p^{\alpha(x)} \big)^{-1} p^{\alpha(x) - 1} u_0(x) dp,\ j \in \{0, \pm\}.
\end{equation}
Since $p\mapsto \big( -\Delta + p^{\alpha} \big)^{-1} p^{\alpha - 1} u_0$ can be extended to a holomorphic map of $\bc\setminus (-\infty,0]$ into $L^2(\br^d)$, we find upon arguing as in the derivation of \cite[Theorem 1.1]{ksy} that the solution $u$ to \eqref{eq:vp} reads 
\begin{equation}
\label{eq-b}
u=u^-+u^0+u^+.
\end{equation}
Moreover, as we have 
$$\norm{u^0(t,\cdot)}_{L^2(\br^d)}\leq \int_{-\theta}^{\theta} e^{\cos \beta}\left\|\left(-\Delta+\left(t^{-1} e^{i\beta}\right)^{\alpha}\right)^{-1}\left(t^{-1} e^{i\beta}\right)^{\alpha}u_0\right\|_{L^2(\br^d)}d\beta,$$
for all $t \in (1,+\infty)$, we infer from \eqref{eq:es2} that
\begin{equation}
\label{t1d}
\norm{u^0(t,\cdot)}_{L^2(\br^d)}\leq C\left( \Vert u_0 \Vert_{L^2(\br^d)} + \Vert \widehat{u_0} \Vert_{L^{2s'}(\br^d)} \right)t^{-\left(\alpha_m-\alpha_* \left( 1-\frac{d}{s} \right) \right)}.
\end{equation}
Similarly we get for all $t \in (1,+\infty)$ that
\begin{eqnarray}
\norm{u^\pm(t,\cdot)}_{L^2(\br^d)} & \leq &\int_1^{+\infty}e^{tr\cos \theta} \left\|\left( -\Delta + \left(r e^{\pm i\theta}\right)^{\alpha} \right)^{-1} \left(r e^{\pm i\theta}\right)^{\alpha - 1} u_0\right\|_{L^2(\br^d)}dr \nonumber\\
& &+\int_{t^{-1}}^{1}e^{tr\cos \theta} \left\|\left( -\Delta + \left(r e^{\pm i\theta}\right)^{\alpha} \right)^{-1} \left(r e^{\pm i\theta}\right)^{\alpha - 1} u_0\right\|_{L^2(\br^d)}dr.\label{t1e}
\end{eqnarray}
The first term in the right hand side of \eqref{t1e} is treated by \eqref{eq:es3}:
\begin{eqnarray}
& & \int_1^{+\infty}e^{tr\cos \theta} \left\|\left( -\Delta + \left(r e^{\pm i\theta}\right)^{\alpha} \right)^{-1} \left(r e^{\pm i\theta}\right)^{\alpha - 1} u_0\right\|_{L^2(\br^d)}dr \nonumber \\
&\leq & C \left( \int_1^{+\infty}e^{tr\cos \theta} dr \right) \|u_0\|_{L^2(\br^d)} \nonumber \\
&\leq & C t^{-1} \|u_0\|_{L^2(\br^d)} \label{t1f}
\end{eqnarray}
For the second term, we apply \eqref{eq:es2}, getting for all $t \in (1,+\infty)$:
\begin{eqnarray*}
& & \int_{t^{-1}}^{1}e^{tr\cos \theta} \left\|\left( -\Delta + \left(r e^{\pm i\theta}\right)^{\alpha} \right)^{-1} \left(r e^{\pm i\theta}\right)^{\alpha - 1} u_0 \right\|_{L^2(\br^d)}dr  \\
& \leq &C\Big( \Vert u_0 \Vert_{L^2(\br^d)} + \Vert \widehat{u_0} \Vert_{L^{2s'}(\br^d)}\Big)  \left(  \int_{t^{-1}}^{1}e^{tr\cos \theta} r^{\alpha_0-1-\alpha_* \left(1-\frac{d}{s} \right)}dr \right)\\
&\leq & C\Big( \Vert u_0 \Vert_{L^2(\br^d)} + \Vert \widehat{u_0} \Vert_{L^{2s'}(\br^d)}\Big) \left( \int_{t^{-1}}^{+\infty}e^{tr\cos \theta} r^{\alpha_m-1-\alpha_* \left(1-\frac{d}{s}\right)}dr \right)\\
& \leq & C\Big( \Vert u_0 \Vert_{L^2(\br^d)} + \Vert \widehat{u_0} \Vert_{L^{2s'}(\br^d)}\Big) \left( t^{-1} \int_{1}^{+\infty}e^{\tau\cos \theta} \left(\frac{\tau}{t}\right)^{\alpha_m-1-\alpha_* \left(1-\frac{d}{s}\right)} d\tau \right)\\
&\leq & C\Big( \Vert u_0 \Vert_{L^2(\br^d)} + \Vert \widehat{u_0} \Vert_{L^{2s'}(\br^d)}\Big) \left( t^{-\left(\alpha_m-\alpha_* \left(1-\frac{d}{s}\right) \right)} \right).
\end{eqnarray*}
Putting this together with \eqref{t1e}-\eqref{t1f}, we obtain 
$$\norm{u^\pm(t,\cdot)}_{L^2(\br^d)}\leq C\Big( \Vert u_0 \Vert_{L^2(\br^d)} + \Vert \widehat{u_0} \Vert_{L^{2s'}(\br^d)} \Big)t^{-\left(\alpha_m-\alpha_* \left(1-\frac{d}{s}\right)\right)},\ t \in (1,+\infty), $$
and hence
$$\norm{u(t,\cdot)}_{L^2(\br^d)}\leq C\Big( \Vert u_0 \Vert_{L^2(\br^d)} + \Vert \widehat{u_0} \Vert_{L^{2s'}(\br^d)} \Big)t^{-\left( \alpha_m-\alpha_* \left(1-\frac{d}{s}\right)\right)},\ t \in (1,+\infty),
$$
by \eqref{eq-b}-\eqref{t1d}. Since the Fourier transform can be extended to a bounded operator from $L^{\frac{2s'}{2s'-1}}(\br^d)$ into $L^{2s'}(\br^d)$ by the Marcinkiewicz interpolation theorem (see e.g. \cite[Appendix B]{St}), we get \eqref{t1a} for $u_0 \in L^1(\br^d) \cap L^2(\br^d)$, and finally for all $u_0\in L^{\frac{2s'}{2s'-1}}(\br^d)\cap L^2(\br^d)$, by a density argument.

We turn now to proving \eqref{t1b}. We proceed as in the derivation of \eqref{t1a}: for all $t \in (0,1]$ we define $u^j(t,\cdot)$, $j \in \{ \pm, 0 \}$, as in \eqref{eq-u}, and we check that \eqref{eq-b} remains valid. Next we apply \eqref{eq:es3} and get for all $t \in (0,1]$ that
\begin{eqnarray*}
\norm{u^0(t,\cdot)}_{L^2(\br^d)} & \leq & \int_{-\theta}^{\theta} e^{\cos \beta}\left\|\left(-\Delta+\left(t^{-1} e^{i\beta}\right)^{\alpha}\right)^{-1}\left(t^{-1} e^{i\beta}\right)^{\alpha}u_0\right\|_{L^2(\br^d)}d\beta\\
& \leq & C (t^{-1})^{-\alpha_m}(t^{-1})^{\alpha_M}\|u_0\|_{L^2(\br^d)}\\
&\leq & C \|u_0\|_{L^2(\br^d)}t^{\alpha_m-\alpha_M}.
\end{eqnarray*}
Similarly, for each $t \in (0,1]$ we have 
\begin{eqnarray*}
\norm{u^\pm(t,\cdot)}_{L^2(\br^d)}&\leq & \int_{t^{-1}}^{+\infty}e^{tr\cos \theta } \left\|\left( -\Delta + \left(r e^{\pm i\theta}\right)^{\alpha} \right)^{-1} \left(r e^{\pm i\theta}\right)^{\alpha - 1} u_0\right\|_{L^2(\br^d)}dr \\
& \leq & C\|u_0\|_{L^2(\br^d)}\left(\int_{t^{-1}}^{+\infty} e^{tr\cos \theta} r^{-\alpha_m} r^{\alpha_M-1}dr\right)\\
& \leq & C\|u_0\|_{L^2(\br^d)}\left(\int_{0}^{+\infty}e^{tr\cos \theta} r^{\alpha_M-\alpha_m-1}dr\right)\\
&\leq & C\|u_0\|_{L^2(\br^d)}\left(t^{-1} \int_{0}^{+\infty}e^{\tau\cos\theta} \left(\frac{\tau}{t}\right)^{\alpha_M-\alpha_m-1}d\tau\right)\\
&\leq & C \|u_0\|_{L^2(\br^d)}t^{\alpha_m-\alpha_M},
\end{eqnarray*}
so we immediately get \eqref{t1b} from the two above estimates and \eqref{eq-b}. 

\section*{Acknowledgments}
 The work of YK and ES was partially supported by the Agence Nationale de la Recherche under grant ANR-17-CE40-0029.

\end{document}